\documentclass{amsart}

\usepackage{amssymb,latexsym,amsmath,extarrows, mathrsfs, amsthm}
\usepackage{graphicx}

\newtheorem{theorem}{Theorem}[section]
\newtheorem{lemma}[theorem]{Lemma}
\newtheorem{proposition}[theorem]{Proposition}

\newtheorem*{definition*}{Definition}

\newtheorem*{acknowledgement}{Acknowledgement}

\newcommand{\cs}{\mathcal S}
\newcommand{\cm}{\mathcal M}
\newcommand{\cp}{\mathcal F}
\newcommand{\cq}{\mathcal Q}
\newcommand{\ct}{\mathcal T}

\newcommand{\lp}{L^{p}}
\newcommand{\nf}{\infty}

\newcommand{\ZR}{\mathbb{R}}
\newcommand{\ZZ}{\mathbb{Z}}
\newcommand{\ZN}{\mathbb{N}}
\newcommand{\ZC}{\mathbb{C}}

\newcommand{\fm}{{\mathfrak M}}
\newcommand{\hichi}{\raisebox{0.7ex}{\(\chi\)}}

\begin{document}

\title{On bilinear Hilbert transform along two polynomials}
\author{Dong Dong}
\address{Center for Scientific Computation and Mathematical Modeling\\University of Maryland\\College Park, MD 20742, USA}
\email{ddong12@cscamm.umd.edu}

\begin{abstract}
We prove that the bilinear Hilbert transform along two polynomials $B_{P,Q}(f,g)(x)=\int_{\ZR}f(x-P(t))g(x-Q(t))\frac{dt}{t}$ is bounded from $L^p \times L^q$ to $L^r$ for a large range of $(p,q,r)$, as long as the polynomials $P$ and $Q$ have distinct leading and trailing degrees. The same boundedness property holds for the corresponding bilinear maximal function $\cm_{P,Q}(f,g)(x)=\sup_{\epsilon>0}\frac{1}{2\epsilon}\int_{-\epsilon}^{\epsilon} |f(x-P(t))g(x-Q(t))|dt$.
\end{abstract}
\let\thefootnote\relax
\let\thefootnote\relax

\maketitle
\section{Introduction}
\setcounter{equation}0
The Hilbert transform along a curve $\gamma: \ZR\to \ZR^n$ is defined by
\begin{equation} \label{HT}
H_{\gamma}(f)(x):=\int_{\ZR} f(x-\gamma(t))\frac{dt}{t},\,\, f\in\cs(\ZR^n).
\end{equation}
Here $\cs(\ZR^n)$, $n\in \ZN$, denotes the space of Schwartz functions on $\ZR^n$. Stein (\cite{S70}) raised the question that under what condition on $\gamma$ is $H_{\gamma}$ bounded from $\lp(\ZR^n)$ to itself for some $p$. Among many curves, a simple but important two dimensional example is the curve $\gamma_{a,b}(t)=(t^a,t^b)$, where $a,b$ are distinct natural numbers. For this particular type of curve, \eqref{HT} becomes
\begin{equation}
H_{\gamma_{a,b}}(f)(x_1,x_2)=\int_{\ZR}f(x_1-t^a,x_2-t^b)\frac{dt}{t},\,\, f\in\cs(\ZR^2).
\end{equation}
The $L^2$-boundedness of $H_{\gamma_{a,b}}$ was first proved by Fabes \cite{F} and Stein and Wainger \cite{SW70}, using different methods. Nagel et.al. \cite{NRW1,NRW2} obtained the $\lp$-boundedness for $p\in (1,\nf)$. It turns out $\gamma_{a,b}$ is the model curve for the very general ``well-curved'' curves (\cite{SW}). 

The purpose of this article is to investigate a bilinear analogue of $H_{\gamma_{a,b}}$. Given two polynomials $P$ and $Q$ on $\ZR$, define the bilinear Hilbert transform along $P,Q$ by
\begin{equation} \label{BPQ}
B_{P,Q}(f,g)(x):=\int_{\ZR}f(x-P(t))g(x-Q(t))\frac{dt}{t},\,\,f,g\in\cs(\ZR).
\end{equation}
In the above definition, instead of just $t^a$ and $t^b$, two arbitrary polynomials are involved, which provides a more general framework. A natural question is that under what condition on $P$ and $Q$ does $B_{P,Q}$ satisfy any $\lp$ estimates. For this problem, we can assume without loss of generality that both $P$ and $Q$ contain no constant term. There are already some positive results in the literature. For example, when $P$ and $Q$ are distinct linear polynomials, $B_{P,Q}$ is in fact the famous bilinear Hilbert transform, whose boundedness was proved by Lacey and Thiele in a pair of breakthrough papers (\cite{LT97,LT99}). Xiaochun Li \cite{LAP} first studied the case $P(t)=t$, $Q(t)=t^d$, $d\in\ZN$, and showed that $B_{P,Q}$ is bounded from $L^2\times L^2$ to $L^1$ (see also \cite{GX,Lie} for some generalizations). Together with Lechao Xiao, Li later (\cite{LX}) obtained the $\lp$ estimates in full range when $P(t)=t$ and $Q$ is any polynomial without linear term. Following the approach in \cite{LAP,LX}, we obtain the theorems below which can be viewed as an extension of Li-Xiao's result to a larger range of pairs of polynomials.
\begin{definition*}
The \textbf{correlation degree} of any two polynomials $P$ and $Q$ is defined as the smallest natural number $d$ such that any non-zero real root of $P'(x)-Q'(x)$ has multiplicity at most $d$.
\end{definition*}

\begin{theorem} \label{thm}
Given two polynomials $P$ and $Q$ without constant terms, we can always write them as
\begin{align}
P(t)=a_{d_1}t^{d_1}+a_{d_1-1}t^{d_1-1}+\dots+a_{e_1}t^{e_1}, 1\le e_1\le d_1, a_{d_1}a_{e_1}\neq 0 \label{poly P}\\
Q(t)=b_{d_2}t^{d_2}+b_{d_2-1}t^{d_2-1}+\dots+b_{e_2}t^{e_2}, 1\le e_2\le d_2, b_{d_2}b_{e_2}\neq 0. \label{poly Q}
\end{align}
Assume $d_1\neq d_2$ and $e_1\neq e_2$. Then there is a constant $C_{P,Q}$ depending on $P$ and $Q$ (and of course $p,q,r$) such that $B_{P,Q}$ defined in \eqref{BPQ} satisfies $\|B_{P,Q}(f,g)\|_r\le C_{P,Q}\|f\|_p\|g\|_q$ for any $f,g\in\cs(\ZR)$, whenever $p,q\in(1,\nf)$, $\frac{1}{r}=\frac{1}{p}+\frac{1}{q}$, $r>\frac{d}{d+1}$. Here $d$ is the correlation degree of $P$ and $Q$. 
\end{theorem}
Remarks.
1. In the expressions \eqref{poly P} and \eqref{poly Q}, we can call $d_1$ and $d_2$ the \textbf{leading degrees}, as they are the degrees of the leading terms. Similarly, $e_1$ and $e_2$ may be called \textbf{trailing degrees} if we name $a_{e_1}t^{e_1}$ and $b_{e_2}t^{e_2}$ as \textbf{trailing terms}. So the condition imposed on $P$ and $Q$ in the theorem can be phrased in words as ``$P$ and $Q$ have distinct leading and trailing degrees''.

2. We conjecture that the constant $C_{P,Q}$ in the theorem may be chosen to be independent of the coefficients of the polynomials. This seems to be a hard and technical problem, whose solution may involve the ideas in the proof of uniform estimate for the bilinear Hilbert transform (\cite{GL,LRev,Thiele02}).

3. For any fixed natural number $d$, there exist polynomials $P$ and $Q$ with correlation degree $d$ such that $B_{P,Q}$ is unbounded whenever $r<\frac{d}{d+1}$ (see Section 3.2 in \cite{LX} for an example). In this sense the lower bound for $r$ given in Theorem \ref{thm} is sharp up to the endpoint. However, if we fix the polynomials $P$ and $Q$, the lower bound of $r$ in Theorem \ref{thm} may not be the best. For instance, let $P(t)=t^6$ and $Q(t)=3t^4-3t^2$. Then $B_{P,Q}$ is the zero operator, which is trivially bounded for $r>\frac{1}{2}$. But the correlation degree of $P$ and $Q$ is 2. It is interesting to find a way to determine the lowest $r$ for any given $P$ and $Q$. This task requires improvement on Lemma \ref{lemma on single scale} in Section \ref{section: reduction} (also see \cite{Dfull} for a recent partial progress on this problem).

4. Some techniques in the study of $B_{P,Q}$ can be used to study discrete analogue of $B_{P,Q}$: see \cite{DCRM,Dthesis,DLS,DM} for some examples.

As a byproduct of the proof of Theorem \ref{thm}, we obtain the same estimate for the bilinear maximal function $\cm_{P,Q}$ defined by
\begin{equation}
\cm_{P,Q}(f,g)(x):=\sup_{\epsilon>0}\frac{1}{2\epsilon}\int_{-\epsilon}^{\epsilon} |f(x-P(t))g(x-Q(t))|\, dt.
\end{equation}
\begin{theorem} \label{thm maximal}
Let $P,Q$ and $p,q,r$ satisfy the conditions stated in Theorem \ref{thm}. Then $\cm_{P,Q}$ is bounded from $L^p \times L^q$ to $L^r$.
\end{theorem}
Just like the relationship between $B_{P,Q}$ and $H_{\gamma_{a,b}}$, $\cm_{P,Q}$ can be viewed as a bilinear analogue of the the maximal function associated with $H_{\gamma_{a,b}}$,
$$
M_{\gamma_{a,b}}(f)(x_1,x_2):=\sup_{h>0}\frac{1}{2h}\int_{-h}^h|f(x_1-t^a,x_2-t^b)|\,dt, \,\,f\in \cs(\ZR^2).
$$
The $L^p$-boundedness of $M_{\gamma_{a,b}}$ was proved in \cite{NRW3} (see \cite{S76,SW76,SW} for further developments on more general curves), and Theorem \ref{thm maximal} is the parallel result in the bilinear setting.

The rest of the paper is organized as follows. In section \ref{section: reduction}, we make careful decompositions on our operator, and after throwing away the paraproduct part, reduce Theorem \ref{thm} to two estimates (Proposition \ref{prop 221} and Proposition \ref{prop pqr}): a scale-type decay estimate when $p=q=2$, and a moderate blow-up estimate for general $p$ and $q$. The decay estimate will be proved in section \ref{section: TT*} and \ref{section: uniformity}, using TT* method and $\sigma$-uniformity method. In the last section, we show how to obtain the moderate blow-up estimate by adapting methods from \cite{LX}, and prove Theorem \ref{thm maximal}.

Throughout the paper we use $C$ to denote a positive constant (which may depend on $P$ and $Q$) whose value is allowed to change from line to line. $A\lesssim B$ means $A\le CB$. $A\simeq B$ is short for $A\lesssim B$ and $B\lesssim A$. We use $A\sim B$ to denote the statement that $B$ is the leading term (principal contribution) of $A$ after using Taylor expansion or stationary phase method. $\hichi_E$ will be used to denote the indicator function of a set $E$.

\section{decomposition and reduction} \label{section: reduction} 
\setcounter{equation}0
Pick an odd function $\rho\in\cs(\ZR)$ supported in the set $\{x:|x|\in(\frac{1}{2},2)\}$ with the property that $t^{-1}=\sum_{j\in\ZZ} 2^j\rho(2^jt)$ for any $t\neq 0$. Then we can write $B_{P,Q}(f,g)(x)=\sum_{j\in\ZZ}T_j(f,g)(x)$, where
\begin{align}
T_j(f,g)(x)&:=\int f(x-P(t))g(x-Q(t))2^j\rho(2^jt)\,dt \label{def of Tj}\\
&=\iint \hat{f}(\xi)\hat{g}(\eta)e^{2\pi i(\xi+\eta)x}m_j(\xi,\eta)\,d\xi d\eta, \notag
\end{align}
and
\begin{equation} \label{def of m_j}
m_j(\xi,\eta):=\int 2^j\rho(2^jt)e^{-2\pi i(\xi P(t)+\eta Q(t))}\,dt.
\end{equation}
We first prove that each $T_j$ is bounded.
\begin{lemma} \label{lemma on single scale}
Let $P$ and $Q$ be two arbitrary polynomials. Then each $T_j$ is bounded from $L^p \times L^q$ to $L^r$, whenever $p,q\in(1,\nf)$, $\frac{1}{r}=\frac{1}{p}+\frac{1}{q}$, $r>\frac{d}{d+1}$, where $d$ is the correlation degree of $P$ and $Q$.
\end{lemma}
\begin{proof}
We only consider the operator $T_0$, as the other cases are similar. The idea of the proof is based on Lemma 9.1 in \cite{LAP}. Note that when $r\ge 1$ the boundedness of $T_0$ follows from Minkowski inequality. So we assume now $r<1$. Since $|t|\simeq 1$, we can restrict $x$ and the support of $f$ and $g$ to a bounded interval $I_{P,Q}$. When the Jacobian $Q'(t)-P'(t)\neq 0$ for all $t$ in the support of $\rho$, $T_0$ is bounded from $L^1\times L^1$ to $L^1$ by changing variables $u=x-P(t)$ and $v=x-Q(t)$. Thus $T_0$ is bounded from $L^1\times L^1$ to $L^{\frac{1}{2}}$ by Cauchy-Schwarz inequality. 

Now we focus on the case that there is a root of $Q'(t)-P'(t)$ lying in the support of $\rho$. Let $t_0$ be such a root and $I(t_0)$ be a small neighborhood of $t_0$. It suffices to prove that
\begin{equation} \label{goal 1 for single scale}
\int_{I_{P,Q}}\left|\int_{I(t_0)}f(x-P(t))g(x-Q(t))\rho(t)\, dt\right|^r\,dx\lesssim\|f\|_p^r\|g\|_q^r,
\end{equation}
for $p,q\in(1,\nf)$, $\frac{1}{p}+\frac{1}{q}=\frac{1}{r}$, $r>\frac{d}{d+1}$. Because of the restriction on $I(t_0)$, the function $\rho$ in \eqref{goal 1 for single scale} can be dropped. Let $\rho_0$ be a bump function supported in $\{t: |t|\in(\frac{1}{2}, 2)\}$ and satisfies $\sum_{j\in\ZZ}\rho_0(2^jt)=1$ for all $t\in\ZR$. Then \eqref{goal 1 for single scale} will be proved once we can show that there is some $\epsilon>0$ such that
\begin{equation} \label{goal 2 for single scale}
\int_{I_{P,Q}}\left|\int f(x-P(t))g(x-Q(t))\rho_0(2^j(t-t_0))\, dt\right|^r\,dx\lesssim 2^{-\epsilon j}\|f\|_p^r\|g\|_q^r
\end{equation}
holds for all large positive $j$.
Changing variable $t-t_0\to t$ and translating $f$ and $g$ by $P(t_0)$ and $Q(t_0)$ respectively, \eqref{goal 2 for single scale} becomes
\begin{equation}
\int_{I_{P,Q}}\left|\int f(x-P_1(t))g(x-Q_1(t))\rho_0 (2^j t)\, dt\right|^r\,dx\lesssim 2^{-\epsilon j}\|f\|_p^r\|g\|_q^r,
\end{equation}
where $P_1(t):=P(t+t_0)-P(t_0)$ and $Q_1(t):=Q(t+t_0)-Q(t_0)$. By the support of $\rho_0$, $|t|\simeq 2^{-j}$. This implies that $P_1(t)\lesssim 2^{-j}$ and $Q_1(t)\lesssim 2^{-j}$ by mean value theorem. So we can for free restrict $x$ to an interval of length $\simeq 2^{-j}$. Let $I_N$ be such an interval and define
$$
T_N(f,g)(x)=\hichi_{I_N}(x)\int f(x-P_1(t))g(x-Q_1(t))\rho_0 (2^j t)\, dt.
$$

It remains to show 
\begin{equation} \label{goal 3 for single scale}
\|T_N(f,g)\|_r\lesssim 2^{-\epsilon j}\|f\|_p\|g\|_q.
\end{equation}
By Fubini theorem, $T_N$ is bounded with norm $\lesssim 2^{-j}$ when $r=1$. Next we aim to get a slow increasing $L^1\times L^1\to L^{\frac{1}{2}}$ norm. By Cauchy-Schwarz inequality,
\begin{equation} \label{1/2 norm}
\int|T_N(f,g)(x)|^{\frac{1}{2}}\,dx\lesssim 2^{-j/2}\|T_N(f,g)\|_1^{\frac{1}{2}}.
\end{equation}
$\|T_N(f,g)\|_1$ can be calculated by changing variables $u=x-P_1(t)$ and $v=x-Q_1(t)$. Using Taylor expansion and the fact that $t_0$ has multiplicity at most $d$, the Jacobian $Q_1'(t)-P_1'(t)$ is bounded below by $2^{-dj}$. Therefore
\begin{equation} \label{1 norm}
\|T_N(f,g)\|_1\lesssim 2^{dj}\|f\|_1\|g\|_1.
\end{equation}
Combining \eqref{1/2 norm} and \eqref{1 norm}, we get
\begin{equation}\label{1/2 final}
\|T_N(f,g)\|_{\frac{1}{2}}\lesssim 2^{(d-1)j}\|f\|_1\|g\|_1.
\end{equation}
Interpolating \eqref{1/2 final} with the $L^1$-norm, we obtain \eqref{goal 3 for single scale}.
\end{proof}

By lemma \ref{lemma on single scale}, to prove Theorem \ref{thm} it suffices to prove the following theorem.
\begin{theorem} \label{thm: large j}
Let $P$ and $Q$ be two polynomials with distinct leading and trailing degrees. Then there is a large $N$ depending on $P$ and $Q$ such that $\sum_{|j|>N}T_j(f,g)(x)$ is bounded from $L^p \times L^q$ to $L^r$ for all $p,q\in (1,\nf)$, $\frac{1}{r}=\frac{1}{p}+\frac{1}{q}$.
\end{theorem}

From the definition \eqref{def of Tj}, we see that $j>N$ corresponds to small $|t|$, in which case the trailing term dominates each polynomial; $j<-N$ corresponds to large $|t|$, in which case $P$ and $Q$ behave almost the same as their leading term. We will only deal with $\sum_{j>N}T_j(f,g)(x)$ since the other case is similar. 

Let $P,Q$ be polynomials written as \eqref{poly P} and \eqref{poly Q}. in When $j$ is large (i.e. $|t|$ is close to $0$), the trailing terms $a_{e_1}t^{e_1}$ and $b_{e_2}t^{e_2}$ dominate $P(t)$ and $Q(t)$, respectively. Since all the constants in our proof are allowed to depend on the coefficients of $P$ and $Q$, we may assume without loss of generality that $a_{e_1}=b_{e_2}=1$. For notation simplicity, from now on we denote $a:=e_1$ and $b:=e_2$. Recall that $e_1\neq e_2$ and thus we may assume $a<b$. With these new notations, we can write $P(t)=t^a+P_{\epsilon}(t)$ and $Q(t)=t^b+Q_{\epsilon}(t)$, where $P_{\epsilon}(t)$ (resp. $Q_{\epsilon}(t)$) consists of terms whose degree is higher than $a$ (resp. $b$). As $P_{\epsilon}(t)$ and $Q_{\epsilon}(t)$ are small when $j>N$ and can be viewed as error terms. We urge the reader to ignore them in the first reading of this paper.

The overall idea of the proof is to look at the size of the symbol $m_j(\xi, \eta)$ defined in \eqref{def of m_j}, which can be estimated by stationary phrase method after proper cut-off and rescaling. By a change of variable, 
\begin{equation} \label{symbol}
m_j(\xi,\eta)=\int \rho(t)e^{-2\pi i\left(\frac{\xi}{2^{aj}}(t^a+\epsilon_P(t))+\frac{\eta}{2^{bj}}(t^b+\epsilon_Q(t))\right)}\, dt,
\end{equation}
where
\begin{align}
\epsilon_P(t):=2^{aj}P_{\epsilon}(2^{-aj}t);\\
\epsilon_Q(t):=2^{bj}Q_{\epsilon}(2^{-bj}t).
\end{align}

Clearly $|\epsilon_P(t)|\le 2^{-N}|t^a|$ and $|\epsilon_Q(t)|\le 2^{-N}|t^b|$ as $j>N$. The expression \eqref{symbol} suggests that we need to consider the sizes of $\frac{\xi}{2^{aj}}$ and $\frac{\eta}{2^{bj}}$. Therefore we choose $\Phi\in\cs(\ZR)$ such that $\hat{\Phi}$ is supported on $\{\xi: |\xi|\in (\frac{1}{2},2)\}$ and 
$$
\sum_{m\in\ZZ}\hat{\Phi}\left(\frac{\xi}{2^m}\right)=1, \,\,\xi\neq 0.
$$
Then decompose $T_j$ as $T_j=\sum_{(m,n)\in\ZZ^2}T_{j,m,n}$ where 
\begin{equation} \label{def of T_jmn}
T_{j,m,n}(f,g)(x):=\iint \hat{f}(\xi)\hat{g}(\eta)e^{2\pi i(\xi+\eta)x}m_j(\xi,\eta)\hat{\Phi}\left(\frac{\xi}{2^{aj+m}}\right)\hat{\Phi}\left(\frac{\eta}{2^{bj+n}}\right)\,d\xi d\eta,
\end{equation}
is the bilinear operator with symbol 
\begin{equation}
\begin{split}
M_{j,m,n}(\xi,\eta):=&m_j(\xi,\eta)\hat{\Phi}\left(\frac{\xi}{2^{aj+m}}\right)\hat{\Phi}\left(\frac{\eta}{2^{bj+n}}\right)\\
=&\int \rho(t)e^{-2\pi i\left(\frac{\xi}{2^{aj}}(t^a+\epsilon_P(t))+\frac{\eta}{2^{bj}}(t^b+\epsilon_Q(t))\right)}\, dt\hat{\Phi}\left(\frac{\xi}{2^{aj+m}}\right)\hat{\Phi}\left(\frac{\eta}{2^{bj+n}}\right)
\end{split}
\end{equation}

In estimating its size, the symbol $M_{j,m,n}$ can be viewed roughly as 
\begin{equation}
\int \rho(t)e^{-2\pi i(2^mt^a+2^nt^b)}\,dt,
\end{equation}
which decays rapidly if $|m-n|$ is large. In fact, $\sum_{j>N}\sum_{|m-n|\gtrsim 1}T_{j,m,n}(f,g)(x)$ can be reduced to the paraproduct studied in \cite{LNY} (see Section 7.2 in \cite{LX} for details). To deal with the remaining $|m-n|\lesssim 1$ case, we can assume without loss of generality that $m=n$. For notation simplicity, denote $M_{j,m}:=M_{j,m,m}$ and $T_{j,m}:=T_{j,m,m}$. Using oddness of $\rho$ and Taylor expansion, $\sum_{j>N}\sum_{m\le 0}T_{j,m}$ can also be reduced to the paraproduct in \cite{LNY}. Thus we will only focus on the most difficult case in proving Theorem \ref{thm: large j}: handing the operator $\sum_{j>N}\sum_{m>0}T_{j,m}$. Our goal is to prove 
\begin{theorem} \label{thm T_jm}
For all $p,q\in (1,\nf)$, $\frac{1}{r}=\frac{1}{p}+\frac{1}{q}$,
\begin{equation}
\left\|\sum_{j>N}\sum_{m>0} T_{j,m}(f,g)\right\|_r\lesssim \|f\|_p\|g\|_q.
\end{equation}
\end{theorem}
By interpolation, the above theorem follows from two propositions below.
\begin{proposition} \label{prop 221}
\begin{equation}
\left\|\sum_{j>N}T_{j,m}(f,g)\right\|_1\lesssim 2^{-\epsilon m}\|f\|_2\|g\|_2 \text{ for some } \epsilon>0.
\end{equation}
\end{proposition}

\begin{proposition} \label{prop pqr}
For $p,q\in (1,\nf)$, $\frac{1}{r}=\frac{1}{p}+\frac{1}{q}$,
\begin{equation}
\left\|\sum_{j>N}T_{j,m}(f,g)\right\|_{r,\nf}\lesssim m\|f\|_p\|g\|_q, 
\end{equation}
\end{proposition}

\section{$TT^*$ method} \label{section: TT*}
\setcounter{equation}0

We prove Proposition \ref{prop 221} in this section and the next. 

Since we can for free insert cut-offs on $\hat{f}$ and $\hat{g}$ according to the support of $M_{j,m}$, in proving Proposition \ref{prop 221} we only need to consider the estimate for a single scale, i.e.
\begin{proposition} \label{prop 221 single scale}
$\|T_{j,m}(f,g)\|_1\lesssim 2^{-\epsilon m}\|f\|_2\|g\|_2$ for any $j>N$ and $m>0$.
\end{proposition}
By rescaling, Proposition \ref{prop 221 single scale} is a consequence of 
\begin{proposition}
For any $j>N$ and $m>0$,
$\|B_{j,m}(f,g)\|_1\lesssim 2^{-\epsilon m}\|f\|_2\|g\|_2$,
where
\begin{equation}
\begin{split}
B_{j,m}(f,g)(x):=&2^{-\frac{(b-a)j}{2}}\int \rho(t)f*\Phi\left(\frac{x}{2^{(b-a)j}}-2^m(t^a+\epsilon_P(t))\right)\\
&\qquad g*\Phi(x-2^m(t^b+\epsilon_Q(t))) \, dt
\end{split}
\end{equation}
\end{proposition}

This proposition follows from the two estimates below.
\begin{proposition} \label{prop 1}
$\|B_{j,m}(f,g)\|_1\lesssim 2^{\frac{(b-a)j-m}{8}}\|f\|_2\|g\|_2$ for any $j>N$ and $m>0$.
\end{proposition}

\begin{proposition} \label{prop 2}
There exists a positive $\delta$ such that
$\|B_{j,m}(f,g)\|_1\lesssim 2^{-\epsilon m}\|f\|_2\|g\|_2$ whenever $(b-a)j>(1-\delta)m$.
\end{proposition}
Proposition \ref{prop 1} is efficient when $m$ is large and Proposition \ref{prop 2} is useful for small $m$. The proofs for the above two propositions require different methods.

We prove Proposition \ref{prop 1} in this section, using a $TT^*$ method. More precisely, we aim to obtain a $L^2\times L^2 \to L^2$ bound with good decay. By making suitable partitions in time spaces, we see that $x$ can be assumed to be supported in an interval of length $\simeq 2^{(b-a)j+m}$. This observation indicates that it suffices to prove 
\begin{equation} \label{goal for TT*}
\|B_{j,m}(f,g)\|_2\lesssim 2^{\frac{(b-a)j-m}{6}} 2^{-\frac{(b-a)j+m}{2}}\|f\|_2\|g\|_2.
\end{equation}
Rewrite $B_{j,m}$ as
\begin{equation}
B_{j,m}(f,g)(x)=2^{-\frac{(b-a)j}{2}}\iint \hat{f}(\xi)\hat{g}(\eta)e^{2\pi i\left(\frac{\xi}{2^{(b-a)j}}+\eta\right)x} I_{\rho,m}\hat{\Phi}(\xi)\hat{\Phi}(\eta)\,d\xi\,d\eta,
\end{equation}
where 
\begin{equation} \label{I}
I_{\rho,m}:=\int \rho(t)e^{-2\pi i 2^m(\xi(t^a+\epsilon_P(t))+\eta(t^b+\epsilon_Q(t)))}\,dt
\end{equation}
Let $\varphi(t):=\xi(t^a+\epsilon_P(t))+\eta(t^b+\epsilon_Q(t))$ and $t_0$ be a solution of $\varphi'(t)=0$. Let $\phi(\xi,\eta):=\varphi(t_0)$. By stationary phase method, 
\begin{equation}
I_{\rho,m}(\xi,\eta)\hat{\Phi}(\xi)\hat{\Phi}(\eta)\sim 2^{-\frac{m}{2}}e^{i 2^m\phi(\xi,\eta)}.
\end{equation} 
Thus we can regard $B_{j,m}$ as 
\begin{equation}
B_{j,m}(f,g)(x)\sim 2^{-\frac{(b-a)j}{2}}2^{-\frac{m}{2}}\iint \hat{f}(\xi)\hat{g}(\eta)e^{2\pi i(\frac{\xi}{2^{(b-a)j}})x}e^{i 2^m\phi (\xi,\eta)}\,d\xi d\eta.
\end{equation}
Then
\[
\begin{split}
\|B_{j,m}\|_2^2=&\int B_{j,m}(x)\overline{B_{j,m}(x)}\,dx \\
=&2^{-(b-a)j-m}\iiiint\limits_{\frac{\xi}{2^{(b-a)j}}+\eta=\frac{\xi_1}{2^{(b-a)j}}+\eta_1} \hat{f}(\xi)\hat{\Phi}(\xi)\overline{\hat{f}(\xi_1)\hat{\Phi}(\xi_1)}\\
&\qquad\qquad\qquad\qquad \hat{g}(\eta)\hat{\Phi}(\eta)\overline{\hat{g}(\eta_1)\hat{\Phi}(\eta_1)}e^{i2^m[\phi(\xi,\eta)-\phi(\xi_1,\eta_1)]}d\xi d\eta d\xi_1 d\eta_1\\
=&2^{-(b-a)j-m}\iiint F_{\tau}(\xi)G_{\tau}(\eta)e^{i 2^m Q_{\tau}(\xi,\eta)}\,d\xi d\eta d\tau,
\end{split}
\]
where
\[
\begin{split}
&F_{\tau}(\xi):=\hat{f}(\xi)\hat{\Phi}(\xi)\overline{\hat{f}(\xi-\tau)\hat{\Phi}(\xi-\tau)}\\
&G_{\tau}(\eta):=\hat{g}(\eta)\hat{\Phi}(\eta)\overline{\hat{g}\left(\eta+\frac{\tau}{2^{(b-a)j}}\right)\hat{\Phi}\left(\eta+\frac{\tau}{2^{(b-a)j}}\right)} \\
&Q_{\tau}(\xi,\eta):=\phi(\xi,\eta)-\phi\left(\xi-\tau,\eta+\frac{\tau}{2^{(b-a)j}}\right).
\end{split}
\]

We claim that whenever $\xi,\eta,\xi-\tau,\eta+\frac{\tau}{2^{(b-a)j}}\in $ supp$\hat{\Phi}$, we have
\begin{equation} \label{claim}
|\partial_{\xi}\partial_{\eta} Q_{\tau}(\xi,\eta)|\gtrsim |\tau|
\end{equation}

Let's briefly justify \eqref{claim}. By the definition of $Q_{\tau}$ and mean value theorem, we need to show that $|\partial_{\xi}^2\partial_{\eta}\phi(\xi,\eta)|$ and $|\partial_{\xi}\partial_{\eta}^2\phi(\xi,\eta)|$ are bounded below by some positive $C$. Let $t_0$ be a root of $F_0(t):=\varphi'(t)$=$D_t(\xi(t^a+\epsilon_P(t))+\eta(t^b+\epsilon_Q(t)))$, and $t_1$ be a root of $F_1(t):=D_t(\xi t^a+\eta t^b)$. Let  $\phi^*(\xi,\eta):=\xi t_1^a+\eta t_1^b=C\left(\frac{\xi^b}{\eta^a}\right)^{\frac{1}{b-a}}$. Then
$$
\phi(\xi,\eta)=\varphi(t_0)=\phi^*(\xi,\eta)+Err(\xi,\eta),
$$
where $Err(\xi,\eta):=\xi(t_0^a-t_1^a)+\eta(t_0^b-t_1^b)+\xi \epsilon_P(t_0)+\eta \epsilon_Q(t_0)$. Clearly the mixed derivatives of $\phi^*(\xi,\eta)$ are bounded below by some positive $C$. It remains to show that $|Err(\xi,\eta)|\le C^{-1}$ for some large $C$. Since $F_0$ and $F_1$ are ``close'', the difference of their inverses $t_0-t_1$ (and its derivatives) is also very small (see Definition A.1 and Lemma A.2 in \cite{LX} for details). By this observation and the facts that $|\epsilon_P(t_0)|$ and $|\epsilon_Q(t_0)|$ are tiny when $N$ is large enough, we conclude that $|Err(\xi,\eta)|$ is very small compared with $1$. This finishes the justification of \eqref{claim}.

By \eqref{claim} and H\"omander principle (Theorem 1.1 in \cite{H}),
\begin{equation}
\iint F_{\tau}(\xi)G_{\tau}(\eta)e^{i 2^m Q_{\tau}}\,d\xi d\eta \lesssim \min\{\|f\|_2^2\|g\|_2^2, 2^{-\frac{m}{2}}|\tau|^{-\frac{1}{2}}\|F_{\tau}\|_2\|G_{\tau}\|_2 \}.
\end{equation}
Therefore,
\[
\begin{split}
\|B_{j,m}\|_2^2&\lesssim 2^{-(b-a)j-m}\int \min\{\|f\|_2^2\|g\|_2^2, 2^{-\frac{m}{2}}|\tau|^{-\frac{1}{2}}\|F_{\tau}\|_2\|G_{\tau}\|_2 \}\,d\tau \\
&\lesssim 2^{-(b-a)j-m}\left[\int_{|\tau|<\tau_0}\|f\|_2^2\|g\|_2^2\, d\tau +\int_{\tau_0 \le |\tau|\lesssim 1}2^{-\frac{m}{2}}|\tau|^{-\frac{1}{2}}\|F_{\tau}\|_2\|G_{\tau}\|_2\right] \\
&\lesssim 2^{-(b-a)j-m}\left[\tau_0 \|f\|_2^2\|g\|_2^2 +2^{-\frac{m}{2}}|\tau_0|^{-\frac{1}{2}}\|F_{\tau}\|_2\|G_{\tau}\|_2\right]\\
&\lesssim 2^{-(b-a)j-m}2^{\frac{(b-a)j-m}{3}}\|f\|_2^2\|g\|_2^2,
\end{split}
\]
from which \eqref{goal for TT*} follows. 

\section{$\sigma$-uniformity method} \label{section: uniformity}
\setcounter{equation}0
We prove Proposition \ref{prop 2} and hence finish the proof of Proposition \ref{prop 221} in this section. Let $I\subseteq \ZR$ be a fixed interval. Let $\sigma\in (0,1]$ and $\cq$ be a collection of real-valued functions.
\begin{definition*}
A function $f\in L^2(I)$ is called $\sigma$-uniform in $\cq$ if $$\left|\int_I f(\xi)e^{-i q(\xi)}\,d\xi\right|\le \sigma \|f\|_{L^2(I)}$$
for all $q\in\cq$.
\end{definition*}
The main tool of proving Proposition \ref{prop 2} is the following theorem, whose proof can be found in Theorem 6.2 in \cite{LAP}.
\begin{theorem} \label{uniformity}
Let $L$ be a bounded sub-linear functional from $L^2(I)$ to $\ZC$. Let $S_{\sigma}$ be the set of all $L^2$ functions that are $\sigma$-uniform in $\cq$ and $U_{\sigma}:=\sup\limits_{f\in S_{\sigma}} \frac{|L(f)|}{\|f\|_{L^2(I)}}$. Then for all functions $f\in L^2(I)$,
\begin{equation}
|L(f)|\lesssim \max\left\{U_{\sigma},\frac{Q_0}{\sigma} \right\}\|f\|_{L^2(I)},
\end{equation}
where $Q_0:=\sup\limits_{q\in\cq}L(e^{iq})$.
\end{theorem}

Now we start to estimate $U_{\sigma}$. Recall that we can assume $x$ is restricted in an interval of length $\simeq 2^{(b-a)j+m}$. We fix such an interval and partition it into $2^m$ intervals of length $\simeq 2^{(b-a)j}$, which are denoted by 
$$
I_k=[\alpha_k-2^{(b-a)j}, \alpha_k+2^{(b-a)j}], k=1,2,\dots, 2^m.
$$
To each $I_k$ we assign an enlarged interval
$$
I_k'=[\alpha_k-C(2^{(b-a)j}+2^m), \alpha_k+C(2^{(b-a)j}+2^m)]
$$
such that $x-2^m(t^b+\epsilon_Q(t))\in I_k'$ whenever $x\in I_k$ and $t\in$ supp$\rho$. So $B_{j,m}$ can be partitioned accordingly as
\[
\begin{split}
B_{j,m}(f,g)(x)=2^{-\frac{(b-a)j}{2}}\sum_{k=1}^{2^m}\hichi_{I_k}(x)&\int f*\Phi\left(\frac{x}{2^{(b-a)j}}-2^m(t^a+\epsilon_P(t))\right)\\
&\hichi_{I_k'} g*\Phi(x-2^m(t^b+\epsilon_Q(t))) \rho(t)\, dt\\
=2^{-\frac{(b-a)j}{2}}\sum_{k=1}^{2^m}\hichi_{I_k}(x)&\iint \hat{f}(\xi)\hat{\Phi}(\xi)e^{2\pi i(\frac{\xi}{2^{(b-a)j}}+\eta)x}\hat{g_k}(\eta)I_{\rho, m}(\xi,\eta)\,d\xi d\eta,
\end{split}
\]
where
$$
g_k(x):=\hichi_{I_k'}g*\Phi(x),
$$
and $I_{\rho,m}$ is defined as \eqref{I}. Since $|\xi| \simeq 1$, $I_{\rho,m}$ has rapid decay unless $|\eta|\simeq 1$. So we may insert a cut-off function $\hat{\Phi}(\eta)$ for free in the above integrand. 

Pair $B_{j,m}$ with an $h\in L^{\infty}$,
\[
\begin{split}
\langle B_{j,m}(f,g),h \rangle=2^{-\frac{(b-a)j}{2}}\sum_{k=1}^{2^m}&\int \hichi_{I_k}(x)h(x)e^{2\pi i\eta x} \\
& \iint \hat{f}(\xi)\hat{\Phi}(\xi)e^{2\pi i\frac{\xi}{2^{(b-a)j}}x}\hat{\Phi}(\eta)\hat{g_k}(\eta)I_{\rho, m}(\xi,\eta)\,d\xi d\eta.
\end{split}
\]
Thanks to the cut-off $\hichi_{I_k}(x)$, and we can replace $e^{2\pi i\frac{\xi}{2^{(b-a)j}}x}$ with $e^{2\pi i\frac{\xi}{2^{(b-a)j}}\alpha_k}$ using Taylor expansion. Thus essentially,
\[
\langle B_{j,m},h\rangle\sim 2^{-\frac{(b-a)j}{2}}\sum_{k=1}^{2^m}\int \check{h_k}(\eta)\Gamma_k(\eta)\hat{g_k}(\eta)\,d\eta,
\]
where
\[
\begin{split}
&h_k(x):=\hichi_{I_k}h(x), \text{ and}\\
&\Gamma_k(\eta):=\hat{\Phi}(\eta)\int \hat{f}(\xi)\hat{\Phi}(\xi)e^{2\pi i\frac{\xi}{2^{(b-a)j}}\alpha_k} I_{\rho,m}(\xi,\eta)d\xi.
\end{split}
\]
As before, we can replace $I_{\rho,m}$ with $2^{-\frac{m}{2}}e^{i 2^m \phi(\xi,\eta)}$ and thus
$$
\Gamma_k(\eta)\sim 2^{-\frac{m}{2}}\hat{\Phi}(\eta)\int \hat{f}(\xi)\hat{\Phi}(\xi)e^{i(2^m\phi(\xi,\eta)+\frac{\xi}{2^{(b-a)j}}\alpha_k)}d\xi.
$$
Let $\cq:=\{A(\xi^{\frac{b}{b-a}}+\epsilon(\xi))+B\xi\}$, where $A,B\in\ZR$, $|A|\simeq a^m$, and $\epsilon(\xi)$ and its derivatives are $\lesssim 2^{-CN}$. Then $2^m\phi(\xi,\eta)+\frac{\xi}{2^{(b-a)j}}\alpha_k \in\cq$ for large $N$. Let $\hat{f}$ be $\sigma$-uniform in $\cq$. Then 
$$
\|\Gamma_k\|_{\nf}\lesssim 2^{-\frac{m}{2}}\sigma\|f\|_2,
$$
and thus
\begin{equation}\label{U sigma} 
\begin{split}
\langle B_{j,m}(f,g),h\rangle &\lesssim 2^{-\frac{(b-a)j}{2} }\sum_{k=1}^{2^m}\|\Gamma_k\|_{\nf}\|h_2\|_2\|g_k\|_2 \\
&\lesssim \sigma\|f\|_2\|h\|_{\nf}(\sum_{k}\|g_k\|_2^2)^{\frac{1}{2}}\\
&\lesssim \begin{cases}
\sigma\|f\|_2\|g\|_2\|h\|_{\nf}  \text{ when } (b-a)j\ge m \\
\sigma 2^{\frac{m-(b-a)j}{2}}\|f\|_2\|g\|_2\|h\|_{\nf} \text{ when } (b-a)j\le m. \\
\end{cases}
\end{split}
\end{equation}
This finishes the computation of $U_{\sigma}$.

Now we turn to $Q_0$. Let $\hat{f}(\xi)=e^{i q(\xi)}$ for some $q\in\cq$. Let $h\in L^{\nf}$ be a function supported on an interval of length $\simeq 2^{(b-a)j+m}$ as before. Define
\[
\begin{split}
\Lambda_q(g,h):=\langle B_{j,m},h\rangle&\\
=2^{-\frac{(b-a)j}{2} }&\iiint\hat{\Phi}(\xi)e^{i(A(\xi^{\frac{b}{b-a}}+\epsilon(\xi))+B\xi)} e^{i\xi \left(\frac{x}{2^{(b-a)j}}-2^m(t^a+\epsilon_P(t))\right)}\,d\xi\\
&\qquad g*\Phi(x-2^m(t^b+\epsilon_Q(t))) \rho(t)\, dtdx.
\end{split}
\]
Our goal is to show
\begin{equation} \label{goal for lambda}
|\Lambda_q(g,h)|\lesssim 2^{-\epsilon m}\|g\|_2\|h\|_{\nf}.
\end{equation}
This means that $Q_0\lesssim 2^{-\epsilon m}$. Combining this with \eqref{U sigma}, Proposition \ref{prop 2} will be proved by Theorem \ref{uniformity}.

To prove \eqref{goal for lambda}, we will use the strategy similar to the previous cases: rescaling, stationary phase, and (local) TT*. Let
$$
|\tilde{\Lambda_q}(g,h)|:=\iint \cp(y,t)g*\tilde{\Phi}\left(y-\frac{t^b+\epsilon_Q(t)}{2^{(b-a)j}}\right)\rho(t)dt h(y)dy,
$$
where $\hat{\tilde{\Phi}}(\xi):=\hat{\Phi}\left(\frac{\xi}{2^{(b-a)j+m}}\right)$ and
$$
\cp(y,t):=2^{\frac{m}{2}}\int \hat{\Phi}(\xi)e^{iA\left(\xi^{\frac{b}{b-a}}+\epsilon(\xi) +\frac{2^m}{A}(y-(t^a+\epsilon_P(t))+\frac{B}{2^m})\xi \right)}\,d\xi.
$$
By rescaling, \eqref{goal for lambda} follows from the estimate
\begin{equation}
|\tilde{\Lambda_q}(g,h)|\lesssim 2^{-\epsilon m}\|g\|_2\|h\|_{\nf} 
\end{equation}
for any $h\in L^{\nf}$ supported in an interval of length $\simeq 1$.

Write 
$$
\cp(y,t)=2^{\frac{m}{2}}\int \hat{\Phi}(\xi)e^{iA\left(\xi^{\frac{b}{b-a}}+\epsilon(\xi) +C'(y-(t^a+\epsilon_P(t))+B')\xi \right)}\,d\xi,
$$
where $C':=\frac{2^m}{A}\simeq 1$ and $B':=\frac{B}{2^m}$.
For simplicity we drop $C'$ from now on. Let $\zeta(z)$ be the solution of $(\xi^{\frac{b}{b-a}}+\epsilon(\xi)+z\xi )'=0$ and $\beta(z):=\zeta(z)^{\frac{b}{b-a}}+\epsilon(\zeta(z))+z\zeta(z)$. Then stationary phase methods gives that 
$$
\cp(y,t)\sim e^{iA\beta(y-(t^a+\epsilon_P(t))+B')}\hat{\Phi}(\zeta(y-(t^a+\epsilon_P(t))+B')).
$$
Since the term $\hat{\Phi}(\zeta(y-(t^a+\epsilon_P(t))+B'))$ can be dropped by Fourier expansion, we have
$$
\tilde{\Lambda_q}(g,h)\sim \iint e^{iA\beta(y-(t^a+\epsilon_P(t))+B')}g*\tilde{\Phi}\left(y-\frac{t^b+\epsilon_Q(t)}{2^{(b-a)j}}\right)\rho(t)dt h(y)dy.
$$
This finishes the use of the stationary phase method. 
The last step is to use TT* method to obtain the decay. Change variable $s=t^b+\epsilon_Q(t)$. Define three new functions $\kappa$, $l$ and $\tilde{\rho}$ by $t=\kappa(s)$, $l(s)=\kappa(s)^a+\epsilon_P(\kappa(s))$ and $\tilde{\rho}(s)ds=\rho(t)dt$. Then
\[
\begin{split}
\tilde{\Lambda_q}(g,h)&=\iint e^{iA\beta(y-l(s)+B')}g*\tilde{\Phi}\left(y-\frac{s}{2^{(b-a)j}}\right)\tilde{\rho}(s)ds\,h(y)dy\\
&\lesssim \|\Delta(h)\|_2\|g\|_2,
\end{split}
\]
where
$$
\Delta(h)(y):=\int e^{iA\beta(y+\frac{s}{2^{(b-a)j}}-l(s)+B')}h\left(y+\frac{s}{2^{(b-a)j}}\right) \tilde{\rho}(s)\,ds.
$$
It remains to show
\begin{equation} \label{goal for delta}
\|\Delta(h)\|_2^2\lesssim 2^{-\epsilon m}\|h\|_{\nf}^2.
\end{equation}
A straightforward calculation gives
\begin{equation} \label{final express for delta}
\|\Delta(h)\|_2^2=\iiint e^{iAO_{\tau}(u,v)}H_{\tau}(u)\Theta_{\tau}(v)dudv\,d\tau,
\end{equation}
where
\[
\begin{split}
&H_{\tau}(u):=h(u)h\left(u+\frac{\tau}{2^{(b-a)j}}\right),\\
&\Theta_{\tau}(v):=\tilde{\rho}(v)\tilde{\rho}(v+\tau),\\
\end{split}
\]
and
$$
O_{\tau}(u,v):=\beta(u-l(v)+B')-\beta\left(u+\frac{\tau}{2^{(b-a)j}}-l(v+\tau)+B'\right).
$$
By the same idea in the proof of \eqref{claim}, we see that the mixed partial derivatives of  $O_{\tau}(u, v)$ is bounded below by $C|\tau|$.
By the operator version of van der Corput lemma (see for example Lemma 5.8 in \cite{LX}), we have
\begin{equation} \label{inner integral estimate}
\iint e^{iAO_{\tau}(u,v)}H_{\tau}(u)\Theta_{\tau}(v)dudv\lesssim \min\{1, |2^m\tau|^{-\epsilon }\}\|H_{\tau}\|_2\|\Theta_{\tau}\|_2.
\end{equation}
By definitions, it is easy to see that $\|H_{\tau}\|_2\lesssim \|h\|_{\nf}^2$ and $\|\Theta_{\tau}\|_2\lesssim 1$. So we can break the integral against $\tau$ in \eqref{final express for delta} into two parts as before: $|\tau|\le \tau_0$ and $\tau_0<|\tau|\lesssim 1$, and use the estimate \eqref{inner integral estimate} to obtain the desired result \eqref{goal for delta}.

\section{$L^r$ estimates and the maximal function} \label{section: Lr and maximal}
\setcounter{equation}0
We start to prove Proposition \ref{prop pqr} and thus finish the proof of Theorem \ref{thm}. Rewrite $T_{j,m}$ as 
\begin{equation}
\begin{split}
T_{j,m}&(f,g)(x)=\\
&\int f*\Phi_{aj+m}\left(x-\frac{t^a+\epsilon_P(t)}{2^{aj}}\right)g*\Phi_{bj+m}\left(x-\frac{t^b+\epsilon_Q(t)}{2^{bj}}\right)\rho(t)\,dt,
\end{split}
\end{equation}
where $\Phi_k(x):=2^k\Phi (2^kx)$. 
Let 
\begin{equation}
\begin{split}
&T^m(f,g)(x):= \\
&\sum_{j>N}\int\left|f*\Phi_{aj+m}\left(x-\frac{t^a+\epsilon_P(t)}{2^{aj}}\right)g*\Phi_{bj+m}\left(x-\frac{t^b+\epsilon_Q(t)}{2^{bj}}\right)\rho(t)\right| \,dt.
\end{split}
\end{equation}
It suffices to prove the boundedness of $T^m$ with norm $\lesssim m$. 

Given any measurable sets $F_1, F_2, F_3$ of finite measure, define $$\Omega:=\bigcup_{i=1}^2\left\{x: \fm \hichi_{F_i}>C\frac{|F_i|}{|F_3|}\right\},$$ where $\fm$ denotes the Hardy-Littlewood maximal operator. Let $F_3':=F_3\setminus \Omega$, which has measure no less than $\frac{|F_3|}{2}$ when $C$ is chosen large enough. By standard interpolation, we need to show that 
\begin{equation} \label{T^m}
|\langle T^m(f,g),h\rangle|\lesssim m|F_1|^{\frac{1}{p}}|F_2|^{\frac{1}{q}}|F_3|^{1-\frac{1}{r}},
\end{equation}
for all $|f|\le \hichi_{F_1}$, $|g|\le \hichi_{F_2}$, $p,q\in (1,\nf)$, $\frac{1}{r}=\frac{1}{p}+\frac{1}{q}$. 

We first remove some error terms related with $\Omega$, Define $\Omega_k:=\{x: \text{dist}(x,\Omega^c)\ge 2^{-k}\}$ and let $\psi_k(x)=\hichi_{\Omega_k^c}*\tilde{\psi_k}(x)$, where $\tilde{\psi}\in\cs(\ZR)$ is Fourier supported in $[-2^k, 2^k]$. It turns out that in proving \eqref{T^m} we can replace $T^m(f,g)$ with
\begin{equation}
\begin{split}
(T')^m(f,g(x):=\sum_{j>N}\int\left|\psi_{aj+m}f*\Phi_{aj+m}\left(x-\frac{t^a+\epsilon_P(t)}{2^{aj}}\right)\right| \\
\left|\psi_{bj+m}g*\Phi_{bj+m}\left(x-\frac{t^b+\epsilon_Q(t)}{2^{bj}}\right)\rho(t)\right| \,dt.
\end{split}
\end{equation}
This is because the difference of these two operators has good control. See Lemma 6.3 in \cite{LX}, whose proof is based on a discussion about whether $x-t$ (or $x-\frac{t^b+\epsilon_Q(t)}{2^{bj}}$) belongs to $\Omega$ or not. That proof can be easily modified to include the $x-\frac{t^a+\epsilon_P(t)}{2^{aj}}$ case. So we focus on proving the following variant of \eqref{T^m}, with $T^m$ being replaced by $(T')^m$:
\begin{equation} \label{T^m variant}
|\langle (T')^m(f,g),h\rangle|\lesssim m|F_1|^{\frac{1}{p}}|F_2|^{\frac{1}{q}}|F_3|^{1-\frac{1}{r}}.
\end{equation}

Time-frequency analysis must be employed to prove \eqref{T^m variant}. For any integers $n,j$, define $I_{n,j}:=[2^{-j}n, 2^{-j}(n+1))$. Let $1^*_{n,j}(x):=\hichi_{I_{n,j}}*\theta_{j+m}(x)$, where $\theta_k\in\cs(\ZR)$ is Fourier supported on $[-2^{-10}2^k, 2^{-10}2^k]$. Then
\begin{equation}
\begin{split}
(T')^m(f,g)(x)=\sum_{j>N}\int &\left|\sum_{n\in\ZZ}f_{n,m,j}\left(x-\frac{t^a+\epsilon_P(t)}{2^{aj}}\right)\right|\\
&\left|\sum_{n\in\ZZ}g_{n,m,j}\left(x-\frac{t^b+\epsilon_Q(t)}{2^{bj}}\right)\rho(t)\right|dt,
\end{split}
\end{equation}
where
\begin{align}
f_{n,m,j}(x):=1^*_{n,aj}\psi_{aj+m}f*\Phi_{aj+m}(x); \notag\\
g_{n,m,j}(x):=1^*_{n,aj}\psi_{bj+m}g*\Phi_{bj+m}(x). \notag
\end{align}
Let $S_0:=\{(j,n)\in\ZZ^2: j>N\}$. For any $S\subseteq S_0$, define $S_j:=\{n\in\ZZ: (j,n)\in S\}$ and
\begin{equation}
\begin{split}
\Lambda_S(f,g):=\sum_{j>N}\iint& \left|\sum_{n\in S_j}f_{n,m,j}\left(x-\frac{t^a+\epsilon_P(t)}{2^{aj}}\right)\right|\\
&\left|\sum_{n\in S_j}g_{n,m,j}\left(x-\frac{t^b+\epsilon_Q(t)}{2^{bj}}\right)\right||\rho(t)|\,dtdx
\end{split}
\end{equation}
We aim to prove that for any finite $S\subseteq S_0$,
\begin{equation}
\Lambda_S(f,g)\lesssim m|F_1|^{\frac{1}{p}}|F_2|^{\frac{1}{q}}|F_3|^{1-\frac{1}{r}}, 
\end{equation}
from which \eqref{T^m variant} follows. The strategy is to organize elements in $S$ into union of subsets called maximal trees. On each tree $\ct\in S$, $\Lambda_{\ct}(f,g)$ can be controlled. Let's perform some reductions on $\Lambda_{\ct}(f,g)$ as in \cite{LX}. By a change of variable $u=x-\frac{t^b+\epsilon_Q(t)}{2^{bj}}$,
\begin{equation}
\Lambda_{\ct}(f,g)=\sum_{j>N}\iint \left|\sum_{n\in \ct_j}f_{n,m,j}\left(u-tr(t)\right)\rho(t)\right|\,dt \left|\sum_{n\in \ct_j}g_{n,m,j}\left(u\right)\right|\,du,
\end{equation}
where $tr(t):=\frac{t^a+\epsilon_P(t)}{2^{aj}}-\frac{t^b+\epsilon_Q(t)}{2^{bj}}$. Since $tr(t)\simeq \frac{t^a}{2^{aj}}$, we have
\begin{equation}
\int\left|\sum_{n\in \ct_j}f_{n,m,j}\left(u-tr(t)\right)\rho(t)\right|\,dt\lesssim \fm \left(\sum_{n\in\ct_j}f_{n,m,j}\right)(u).
\end{equation}
From here the translation determined by $t$ disappears and thus we can use the same calculations as in \cite{LX}. We omit the details. This finishes the proof of Proposition \ref{prop pqr} and Theorem \ref{thm}.

Now we show how to use Lemma \ref{lemma on single scale} and Theorem \ref{thm T_jm} to obtain the boundedness of the bilinear maximal function $\cm_{P,Q}$, proving Theorem \ref{thm maximal}.
By triangle inequality, it suffices to consider the following operator
\begin{equation}
T^*(f,g)(x):=\sup_{j\in\ZZ}T_j(f,g)(x),
\end{equation}
where $T_j$ is defined as in \eqref{def of Tj} and $f, g$ are non-negative. 
By Lemma \ref{lemma on single scale} and symmetry, we can further assume that the supremum is taken over $j>N$ for some large $N$.

As before, decompose $T_j=\sum_{(m,n)\in \ZZ^2} T_{j,m,n}$ (see \eqref{def of T_jmn}). Let $E:=\{|m-n|\gtrsim 1\}\cup \{\max\{m,n\}\le 0\}$. Using Fourier expansion and integration by parts (or Taylor expansion), it is easy to see that
$$
\sup_{j>N}\left|\sum_{(m,n)\in E}T_{j,m,n}(f,g)(x)\right|\lesssim \fm f(x)\fm g(x).
$$
By H\"older inequity and the boundedness of $\fm$, $\sup_{j>N}|\sum_{(m,n)\in E}T_{j,m,n}(f,g)|$ is bounded from $L^p\times L^q$ into $L^r$. 

For $(m,n)\in\ZZ^2\setminus E$, we can assume without loss of generality that $m=n$. In this case we bound $\sup_{j>N}|T_{j,m,m}(f,g)(x)|$ crudely by $\sum_{j>N}|T_{j,m,m}(f,g)(x)|$. Since each $\sum_{j>N}|T_{j,m,m}|$ is bounded with $2^{-\epsilon
 m}$ decay in norm by Theorem \ref{thm T_jm}, we conclude that $\sup_{j>N}|\sum_{(m,n)\in \ZZ^2\setminus E}T_{j,m,n}|$ is bounded. This finishes the proof of Theorem \ref{thm maximal}.

\begin{acknowledgement}
The author would like to thank Prof. Xiaochun Li for helpful discussions on this topic. He also acknowledges the support from Gene H. Golub Fund of Mathematics Department at University of Illinois.
\end{acknowledgement}

\bibliographystyle{amsplain}

\begin{thebibliography}{30}

\bibitem{CNSW} M.~Christ, A. ~Nagel, E.M. ~Stein, S.~Wainger,
\emph{Singular and maximal Radon transforms: analysis and geometry},  
Ann. of Math. (2) 150 (1999), no. 2, 489-577.

\bibitem{DCRM} D.~Dong, 
\emph{On a discrete bilinear singular operator},  
C. R. Math. Acad. Sci. Paris 355 (2017), no. 5, 538-542. 

\bibitem{Dthesis} D.~Dong, 
\emph{Multilinear operators in harmonic analysis: methods and applications},  
Thesis (Ph.D.)-University of Illinois at Urbana-Champaign, 2018, http://hdl.handle.net/2142/101503

\bibitem{Dfull} D.~Dong, 
\emph{Full range boundedness of bilinear Hilbert transform along certain polynomials},  
Math. Inequal. Appl., to appear.

\bibitem{DLS} D.~Dong, X.~Li, W.~Sawin 
\emph{Improved estimates for polynomial Roth type theorems in finite fields},  
https://arxiv.org/abs/1709.00080, J. Anal. Math., to appear.

\bibitem{DM} D.~Dong, X.~Meng
\emph{Discrete bilinear Radon transforms along arithmetic functions with many common values},  Bull. Lond. Math. Soc. 50 (2018), no. 1, 132-142. 

\bibitem{F} E.~Fabes,
\emph{Singular integrals and partial differential equations of parabolic type},
Studia Math. 28 1966/1967 81-131.

\bibitem{GL} L.~Grafakos and X.~Li,
\emph{Uniform bounds for the bilinear Hilbert transforms. I},
Ann. of Math. (2) 159 (2004), no. 3, 889-933.

\bibitem{GX} J.~Guo, L.~Xiao,
\emph{Bilinear Hilbert transforms associated with plane curves},
J. Geom. Anal. 26 (2016), no. 2, 967-995.

\bibitem{H}  L.~H\"ormander,
\emph{Oscillatory integrals and multipliers on $FL^p$} 
Ark. Mat. 11(1973), 1-11. .

\bibitem{LT97}  M.~Lacey and C.~Thiele,
\emph{$L^p$ estimates on the bilinear Hilbert transform for $2<p<\infty$,} 
Ann. of Math. (2) 146 (1997), no. 3, 693-724.

\bibitem{LT99} M. ~Lacey and C.~Thiele,
\emph{On Calder\'{o}n's conjecture},
Ann. of Math. (2) 149 (1999), no. 2, 475-496.

\bibitem{LRev} X.~Li,
\emph{Uniform bounds for the bilinear Hilbert transforms. II},
Rev. Mat. Iberoam. 22 (2006), no. 3, 1069-1126.

\bibitem{LNY} X.~Li,
\emph{Uniform estimates for some paraproducts},
New York J. Math. 14 (2008), 145–-192. 

\bibitem{LAP} X.~Li,
\emph{Bilinear Hilbert transforms along curves I: The monomial case},
Anal. PDE 6 (2013), no. 1, 197-220. 

\bibitem{Lie} V.~Lie,
\emph{On the boundedness of the bilinear Hilbert transform along ``non-flat'' smooth curves},
Amer. J. Math. 137 (2015), no. 2, 313-363. 

\bibitem{LX} X.~Li, L.~Xiao,
\emph{Uniform estimates for bilinear Hilbert transform and bilinear maximal functions associated to polynomials},
Amer. J. Math., 138 (2016), No. 4, 907-962.

\bibitem{NRW1} A.~Nagel, N.~Riviere, S.~Wainger,
\emph{On Hilbert transforms along curves},  
Bull. Amer. Math. Soc. 80 (1974), 106-108.

\bibitem{NRW2} A.~Nagel, N.~Riviere, S.~Wainger,
\emph{On Hilbert transforms along curves. II},  
Amer. J. Math. 98 (1976), no. 2, 395-403.

\bibitem{NRW3} A.~Nagel, N.~Riviere, S.~Wainger,
\emph{A maximal function associated to the curve $(t,t^2)$},  
Proc. Nat. Acad. Sci. U.S.A. 73 (1976), no. 5, 1416–-1417.


\bibitem{S70} E.M. ~Stein,
\emph{Some problems in harmonic analysis},  
Proc. Internat. Congress Math (Nice, 1970), vol.1, Gauthier-Villars, Paris, 1971, pp.173--190.


\bibitem{S76} E.M. ~Stein,
\emph{Maximal functions. II. Homogeneous curves}, 
Proc. Nat. Acad. Sci. U.S.A. 73 (1976), no. 7, 2176-2177.

\bibitem{SW70} E.M. ~Stein, S.~Wainger,
\emph{The estimation of an integral arising in multiplier transformations}, 
Studia Math. 35 1970 101--104.

\bibitem{SW76} E.M. ~Stein, S.~Wainger,
\emph{Maximal functions associated to smooth curves},  
Proc. Nat. Acad. Sci. U.S.A. 73 (1976), no. 12, 4295-4296.

\bibitem{SW} E.M. ~Stein, S.~Wainger,
\emph{Problems in harmonic analysis related to curvature},  
Bull. Amer. Math. Soc. 84 (1978), no. 6, 1239-1295.


\bibitem{Thiele02} C.~Thiele,
\emph{A uniform estimate},   
Ann. of Math. (2) 156 (2002), no. 2, 519-563.


\end{thebibliography}

\end{document}